\documentclass{siamart220329}

\usepackage[only,llbracket,rrbracket]{stmaryrd}
\usepackage{lipsum}
\usepackage{amsfonts,amsmath,amssymb}
\usepackage{bm}
\usepackage{mathrsfs}
\usepackage{enumitem}
\usepackage{cases}
\usepackage{graphicx}
\usepackage{stmaryrd}
\usepackage{xparse}
\usepackage{epstopdf}
\usepackage{cite}
\usepackage{algorithmic}
\usepackage{color}
\ifpdf
  \DeclareGraphicsExtensions{.eps,.pdf,.png,.jpg}
\else
  \DeclareGraphicsExtensions{.eps}
\fi

\def\Omega{\varOmega}
\def\Delta{\varDelta}

\def\lj{\llbracket}
\def\rj{\rrbracket}

\NewDocumentCommand{\dgal}{sO{}m}{%
  \IfBooleanTF{#1}
    {\dgalext{#3}}
    {\dgalx[#2]{#3}}%
}

\NewDocumentCommand{\dgalext}{m}{%
  \sbox0{%
    \mathsurround=0pt 
    $\left\{\vphantom{#1}\right.\kern-\nulldelimiterspace$%
  }%
  \sbox2{\{}%
  \ifdim\ht0=\ht2
    \{\kern-.45\wd2 \{#1\}\kern-.45\wd2 \}%
  \else
  \fi
}

\NewDocumentCommand{\dgalx}{om}{%
  \sbox0{\mathsurround=0pt$#1\{$}%
  \sbox2{\{}%
  \ifdim\ht0=\ht2
    \{\kern-.45\wd2 \{#2\}\kern-.45\wd2 \}%
  \else
    \mathopen{#1\{\kern-.5\wd0 #1\{}
    #2
    \mathclose{#1\}\kern-.5\wd0 #1\}}
  \fi
}

\def\Omega{\varOmega}
\def\Delta{\varDelta}

\def\la{\{\hspace*{-0.4em}\{}
\def\ra{\}\hspace*{-0.4em}\}}
\def\eps{\epsilon}

\def\T{{\mathcal{T}}}

\usepackage{subcaption}

\usepackage{multirow}

\graphicspath{{fig/}{fig/2d/}{fig/3d}}

\newsiamremark{remark}{Remark}
\newsiamremark{numscheme}{Numerical Scheme}
\newsiamremark{assumption}{Assumption}
\newsiamremark{hypothesis}{Hypothesis}
\crefname{hypothesis}{Hypothesis}{Hypotheses}
\newsiamthm{claim}{Claim}

\headers{$C^0$IP for Monge-Amp\`ere equations
}{Tianyang Chu, Hailong Guo and Zhimin Zhang}

\title{New $C^0$ interior penalty   method for Monge-Amp\`ere equations}

\author{
Tianyang Chu\thanks{LSEC, Institute of Computational Mathematics and Scientific/Engineering Computing, Academy of Mathematics and Systems Science, Chinese Academy of Sciences, Beijing 100190, China (\email{tchu@lsec.cc.ac.cn}).}
\and Hailong Guo\thanks{School of Mathematics and Statistics, The University of Melbourne, Parkville, VIC 3010, Australia (\email{hailong.guo@unimelb.edu.au}).}
\and Zhimin Zhang\thanks{Department of Mathematics, Wayne State University, Detroit, MI 48202, USA (\email{ag7761@wayne.edu}).}
}

\usepackage{amsopn}

\begin{document}

\maketitle

\begin{abstract}
Monge-Amp\`{e}re equation is a prototype  second-order fully nonlinear  partial differential equation. 
In this paper, we propose a new idea to design and analyze the $C^0$ interior penalty method to approximation the viscosity solution of the   Monge-Amp\`{e}re equation. The new methods is inspired from the discrete Miranda-Talenti estimate.    Based on the vanishing moment representation, we approximate the Monge-Amp\`{e}re equation by the fourth order semi-linear equation with some additional boundary conditions.  We will use the discrete Miranda-Talenti estimates to ensure the well-posedness of the numerical scheme and derive the error estimates. 
\end{abstract}



\section{Introduction}\label{sec: int}
In this paper, we consider the numerical approximation of the following  fully nonlinear  Monge-Amp\`{e}re equation with Dirichlet boundary condition
\begin{subequations}\label{equ: MA equation}
    \begin{align}
        \det(D^2u) &= f \quad \text{in $\varOmega$}, \\
        u &=g\quad \text{on $\partial\varOmega$},
    \end{align}
\end{subequations}
where $\varOmega \subset \mathbb{R}^d$, $d \in \{2,3\}$ is a bounded polygonal convex domain with boundary $\varGamma = \partial\Omega$, and $\det(D^2u)$ denotes the determinant of the Hessian matrix $D^2u$.

The Monge-Amp\`{e}re equation serves as a prototypical example of fully nonlinear partial differential equations \cite{caffarelli1995fullynonlinear}. It arises in many important applications such as differential geometry \cite{guti2016ma, figalli2017mae}, the reflector design problem \cite{wang1996inversedesign}, and optimal transports \cite{villani2003topicinot, benamou2000cfmot}. For the Monge-Amp\`{e}re equation, the classical solution may not exist on a convex domain even though functions $f$ and $g$ are smooth \cite{gilbarg2001secondorderpde}.  Given the fully nonlinear nature of \eqref{equ: MA equation}, traditional weak solution theories based on variational calculus are not directly applicable. Consequently, alternative solution concepts such as Aleksandrov solutions and viscosity solutions have emerged. The viscosity solution theory of Monge-Amp\`{e}re equation  have been well developed in the last half century. For comprehensive insights into these theories and related developments, readers are referred to the monographs \cite{caffarelli1995fullynonlinear, figalli2017mae, guti2016ma, gilbarg2001secondorderpde} and the references therein.

The development of numerical methods for the Monge-Amp\`{e}re equation falls behind its PDE theory. In 1988, Oliker and Prussner \cite{oliker1988geometric} constructed the first finite difference methods to compute the Aleksandrov solutions. The first practical numerical method for the Monge-Amp\`{e}re equation came 20 years later when Oberman proposed the wide stencil finite difference method in 2008 \cite{oberman2008wide} based on the framework of Barles and Souganidis \cite{bartles1991framework}. Since then, there have been extensive advances in the numerical methods for the Monge-Amp\`{e}re equation. Famous examples include filter schemes \cite{froese2013filterscheme}, $L^2$-projection methods \cite{bohmer2008l2projection, brenner2011c0ip2d, brenner2012c0ip3d, neilan2013quad, adetola2022recovery}, least squares methods \cite{brenner2021convex, dean2004leastsquare}, vanishing moment methods \cite{feng2009mixed, feng2009modified, neilan2010nonconforming, chenfengzhang2021, feng2009vanishing, feng2011conforming}, two-scale finite element methods \cite{nochetto2019twoscalepoint, nochetto2019twoscale}, and et al.  For comprehensive lists of numerical methods for the  the Monge-Amp\`{e}re equation, interested readers are referred to the recent review papers \cite{feng2013siamreview, neilan2017review, neilan2020review}.

The main purpose of our paper is to design a new $C^0$ interior penalty (C0IP) method using the vanishing moment approximation of \eqref{equ: MA equation} \cite{feng2009mixed, feng2009modified, neilan2010nonconforming}. For this end, we approximate equation \eqref{equ: MA equation} by the following fourth-order semi-linear partial differential equation (PDE) with an additional boundary condition:
\begin{subequations}\label{equ: model equations}
    \begin{align}
        -\epsilon \varDelta^2 u^{\epsilon} + \det(D^2u^{\epsilon}) = f& \quad \text{in $\varOmega$}, \\
        u^{\epsilon} =g&\quad \text{on $\partial\varOmega$}, \\
        \varDelta u^{\epsilon} =\eps & \quad \text{on $\partial\varOmega$}.
    \end{align}
\end{subequations}


For the PDE theory of \eqref{equ: model equations}, it has been proven in the two-dimensional case \cite{feng2009vanishing} and in the $d$-dimensional ($d\ge 2)$ radially symmetric case \cite{feng2014convergence} that for any $\eps>0$ and $f > 0$, \eqref{equ: model equations} has a unique convex solution $u^{\eps}$. Moreover, $u^{\eps}$ converges uniformly as $\eps \to 0^{+}$, and we have the following {\it a priori} bounds:
\begin{equation}\label{equ: priori bounds for u^eps}
	\begin{aligned}
		&\|u^{\eps}\|_{H^{j}(\Omega)} = \mathcal{O}(\eps^{\frac{1-j}{2}}) \ (j = 2,\; 3), && \|u^{\eps}\|_{W^{j, \infty}(\Omega)} = \mathcal{O}(\eps^{1-j}) \ (j = 1,\; 2), \\
		&\|\Phi^{\eps}\|_{L^2(\Omega)} = \mathcal{O}(\eps^{-\frac{1}{2}}), && \|\Phi^{\eps}\|_{L^{\infty}(\Omega)} = \mathcal{O}(\eps^{-1}), 
	\end{aligned}
\end{equation}
where $\Phi^{\eps}$ denotes the cofactor matrix of $D^2u^{\eps}$.
Based on these findings, we make the following assumption on the solution $u^{\eps}$:
\begin{assumption}\label{ass:convex assumption}
    For problem \eqref{equ: model equations}, we assume that there exists a unique convex solution $u^{\eps}$ such that $u^{\eps} \in H^{3}(\Omega)\cap W^{2,\infty}(\Omega)$, and satisfies \eqref{equ: priori bounds for u^eps}. 
\end{assumption}

The classical finite element methods for fourth-order elliptic PDEs, such as the $C^1$ conforming finite element method\cite{feng2011conforming}, the Morley nonconforming finite element method\cite{neilan2010nonconforming},  mixed finite element method\cite{feng2009mixed},  and the recovery-based linear finite element method\cite{chenfengzhang2021},  have been utilized to solve the semi-linear problem \eqref{equ: model equations}. Compared to the aforementioned finite element methods, the C0IP method using standard Lagrange elements stands out for its flexibility and ease of implementation, making it a preferred choice for high-order PDE solvers\cite{brenner2005c0ipanalysis,engel2002c0ip}. However, extending C0IP methods to address the model problem \eqref{equ: model equations} is not straightforward. Although some numerical results are presented in \cite{neilan2020review}, the convergence analysis of the C0IP method remains an {\em open problem}, as stated in the review paper \cite{neilan2020review}.

In this paper, we try to answer this open question by proposing  a new C0IP method with erorr analysis. 
Our methodology for designing the new C0IP methods is to use the discrete Miranda-Talenti estimate \cite{neilan2019discrete} as the main analytic tool. The discrete Miranda-Talenti estimate was originally used to develop numerical methods for the Hamilton-Jacobi-Bellman (HJB) equation with Cordes coefficients \cite{neilan2019discrete, smears2013discontinuousnd, smears2014discontinuoushjb}. Although the Monge-Amp\'{e}re equation \eqref{equ: MA equation} can be reformulated as an HJB equation \cite{feng2017ma2hjb, gallistl2023ma2hjb}, the numerical methods for the HJB equation \cite{smears2013discontinuousnd, smears2014discontinuoushjb, neilan2019discrete} including the discrete Miranda-Talenti estimate can be used.
To the best of our knowledge, it is the first time the discrete Miranda-Talenti estimate has been used to design a C0IP method to compute the viscosity solution of \eqref{equ: MA equation}. The discrete Miranda-Talenti estimate not only inspires us to design a new C0IP method for the nonlinear equation with linear jumps only on interior edges and averaging terms but also allows us to establish a discrete Sobolev inequality, which is a key ingredient in the fixed-point argument.

Compared to the existing C0IP method for the Monge-Ampere equation in \cite{brenner2011c0ip2d, brenner2012c0ip3d, neilan2013quad}, the new proposed method has several advantages. First, our penalty only involves linear terms, contrasting with the fact that nonlinear penalties are needed to ensure stability for the methods in \cite{brenner2011c0ip2d, brenner2012c0ip3d, neilan2013quad}. Second, our method is designed to compute the viscosity solution, which requires lower regularity for convergence compared to the existing methods that compute classical solutions and may fail for singular solutions. Third, the weak formulation is  much simpler,  especially in 3D cases.


The remaining sections are organized as follows.  In Section \ref{sec: pre}, we introduce the notations for finite element spaces and present the discrete Miranda-Talenti estimate. Section \ref{sec:lin} is dedicated to designing and analyzing a C0IP formulation for the linearized Monge-Amp\`{e}re equation, leveraging the discrete Miranda-Talenti estimate. In Section \ref{sec:c0ip}, we first formulate the new C0IP method; then, we prove well-posedness and establish error estimates for the discrete formulation using the Bownder fixed-point technique. Section \ref{sec:num} demonstrates the performance of the proposed method through a series of numerical examples. Finally, some conclusions are drawn in Section \ref{sec:con}.

\section{Preliminaries}\label{sec: pre}
The purpose of this section is to provide some background material. It begins by introducing related notations, including finite element spaces and the interpolation operator in subsection \ref{ssec:not}. In the following subsection \ref{ssec:dmt}, we present the discrete Miranda-Talenti estimate and use it to prove a discrete Sobolev inequality.

\subsection{Notations}\label{ssec:not}
Let $|D|$ represent the measure of a measurable set $D$. For $s \geq 0$ and $p \in [1,\infty]$, the Sobolev space $W^{s,p}(D)$ is defined as the standard Sobolev space \cite{evans2010} on the domain $D$. Specifically, when $s=0$ or $p=2$, we have $W^{0,p}(D):=L^p(D)$ and $H^{s}(D):= W^{s,2}(D)$. The norm and semi-norm on $W^{s,p}(D)$ are denoted by $\|\cdot\|_{W^{s,p}(D)}$ and $|\cdot|_{W^{s,p}(D)}$ respectively. Additionally, let $H_0^1(D)$ be the subspace of $H^1(D)$ that comprises functions vanishing on $\partial D$. The inner product on $L^2(D)$ is denoted as $(\cdot,\cdot)_{L^2(D)}$. We introduce the following notation:
\begin{equation*}
V:= H^2(\Omega), \quad V^0:= H^2(\Omega)\cap H_0^1(\Omega), \quad V^g:= \{ v \in V; \,  v|_{\partial\Omega}=g \}.
\end{equation*}

Let $\mathcal{T}_h$ be a quasi-uniform, shape-regular,  and conforming simplicial triangulation of $\Omega$ \cite{brenner2008mathematical, ciarlet2002}. For each $K \in \mathcal{T}_h$, we define $h_{K}:= \text{diam}(K)$ and $h:= \max_{K\in \mathcal{T}_h} h_{K}$. The set of interior edges/faces is denoted by $\mathcal{F}_h^{i}$, and $\mathcal{F}_h^{b}$ denotes the set of boundary edges/faces. Similarly, we define  $h_{F}:= \text{diam}(F)$ for any $F\in \mathcal{F}_h^{i}$. 
The jump of a vector-valued function ${\bm w} \in \mathbb{R}^{d}$ on an interior edge/face $F = \partial K^{+} \cap \partial{K}^{-}$ is defined as:
\begin{equation*}
\lj {\bm w} \rj |_{F} = {\bm w}^{+} \cdot {\bm n}_{+}|_{F} + {\bm w}^{-} \cdot {\bm n}_{-}|_{F},
\end{equation*}
where ${\bm n}_{+}$ and ${\bm n}_{-}$ denote the unit outward normals of $K^{+}$ and $K^{-}$, respectively.
Additionally, the average of a piecewise smooth function $w$ across an interior edge/face $F = \partial K^{+} \cap \partial{K}^{-}$ is given by:
\begin{equation*}
\la w \ra |_{F} = \frac{1}{2} \left( w^{+}|_{F} + w^{-}|_{F} \right).
\end{equation*}

Let $\mathbb{P}_k(D)$ denote the space of polynomials of degree less than or equal to $k$ over the domain $D$. The piecewise polynomial space of degree $k$ associated with $\mathcal{T}_h$ is defined as
$$
\mathbb{P}_k(\Omega;\mathcal{T}_h):= \{ v \in L^2(\Omega): v|_{K} \in \mathbb{P}_k(K) \text{ for all } K \in \mathcal{T}_h \}.
$$
The norm $\|\cdot\|_{L^2(\mathcal{T}_h)}$ is defined as
$$
\|v\|_{L^2(\mathcal{T}_h)}^2 = \sum_{K \in \mathcal{T}_h} \|v\|_{L^2(K)}^2.
$$
Similar notation is used for $\|\cdot\|_{H^2(\mathcal{T}_h)}$. Furthermore, we define the inner product $(\cdot,\cdot)_{L^2(\mathcal{T}_h)}$ for the piecewise smooth function space as
$$
(u,v)_{L^2(\mathcal{T}_h)} = \sum_{K \in \mathcal{T}_h} (u,v)_{L^2(K)}.
$$

The standard Lagrange finite element space on $\mathcal{T}_h$ is defined as
\begin{equation*}
V_h = \mathbb{P}_k(\Omega;\mathcal{T}_h) \cap H^1(\Omega), \quad V_h^0 = V_h \cap H_0^1(\Omega).
\end{equation*}
To simplify notation, the elementwise Hessian matrix (Laplacian) for $v_h \in V_h$ is still denoted by $D^2 v_h$ ($\Delta v_h$).
For any function $v \in C^0(\Omega)$, we denote its interpolation in $V_h$ as $v_I$. 
 For the interpolation $v_I$, we have the following approximation property \cite{brenner2008mathematical, ciarlet2002}:
\begin{equation} \label{equ: approximaiton of interpolation}
	\|v-v_I\|_{L^2(\Omega)} + h\|v-v_I\|_{H^1(\Omega)} + h^2\|v-v_I\|_{H^2(\mathcal{T}_h)} \le C h^{\min\{k+1,s\}}\|v\|_{H^s(\Omega)}. 
\end{equation}

Throughout the paper, the letter $C$ or $c$, with or without subscripts, denotes a generic constant that is independent of $h$, $\epsilon$, and the penalty parameter $\sigma$. The value of this constant may vary and might not be the same at each occurrence.

\subsection{Discrete Miranda-Talenti estimate}\label{ssec:dmt}
In this subsection, we introduce the discrete Mirand-Talenti estimate, which serves as the main fundamental tool in designing and analyzing the new $C^0$ interior penalty finite element method.  Before that, we recall the following Miranda-Talenti estimate 

\begin{theorem}[Miranda-Talenti estimate\cite{gris1985, maps2000}]
Suppose $\Omega \subset \mathbb{R}^d$ is a bounded convex domain. Then, for all $v \in H^2(\Omega) \cap H_0^1(\Omega)$, the following inequality holds:
\begin{equation*}
\|D^2v\|_{L^2(\Omega)} \le \|\Delta v\|_{L^2(\Omega)}.
\end{equation*}
\end{theorem}

One crucial ingredient for establishing the discrete Miranda-Talenti estimate is the introduction of an enrichment operator $E_h$. Let $E_h$ denote the enrichment operator defined in \cite{neilan2019discrete}, which maps $v_h \in V^0_h$ to a subspace of $H^2(\Omega)\cap H^1_0(\Omega)$.  With the aid of the enrich operator $E_h$, \cite{neilan2019discrete} proved the following discrete Miranda-Talenti estimate:
\begin{theorem}[Discrete Miranda-Talenti estimates]
Let $\Omega \subset \mathbb{R}^d$ $(d = 2,3)$ be a convex polytope. Then, for any $v_h \in V_h^0$, we have    \begin{equation}\label{equ:discrete MT}
        \|D^2 v_h \|_{L^2(\mathcal{T}_h)} \le \|\Delta v_h \|_{L^2(\mathcal{T}_h)} + C_{\dag} \Bigg(   \sum_{ F \in \mathcal{F}_h^{i}} h_{F}^{-1} \big\|   \lj \nabla v_h \rj    \big\|_{L^2(F)}^2     \Bigg)^{\frac{1}{2}}, 
    \end{equation}
    for some constant $C_{\dag}$ independent of $h$ and $v_h$. 
\end{theorem}

Noticing that $|E_hv_h|_{H^1(\Omega)} \le C |E_hv_h|_{H^2(\Omega)}$, following the proof of \cite[Lemma 3 and Theorem 1]{neilan2019discrete}, we can show that
\begin{equation}
    \|\nabla v_h \|_{L^2(\mathcal{T}_h)} \le C \left( \|\Delta v_h \|_{L^2(\mathcal{T}_h)} +  \Bigg(   \sum_{ F \in \mathcal{F}_h^{i}} h_{F}^{-1} \big\|   \lj \nabla v_h \rj    \big\|_{L^2(F)}^2     \Bigg)^{\frac{1}{2}} \right), 
\end{equation}
for some constant $C$. It triggers us to define a new mesh-dependent norm on $V_h^0$ as
\begin{equation}
    \|v_h\|_h^2 := \|D^2 v_h \|_{L^2(\mathcal{T}_h)}^2 + \sum_{F \in \mathcal{F}_h^{i}} h_{F}^{-1} \big\|   \lj \nabla v_h \rj    \big\|_{L^2(F)}^2. 
\end{equation}
It is not hard to see that 
\begin{equation}
	\|v_h\|_{H^1(\Omega)} \le C \|v_h\|_h. 
\end{equation}

Armed with the new norm, we can establish the following discrete Sobolev inequality:
\begin{theorem}[Discrete Sobolev inequality]
    For any $v_h \in V_h^0$, we have 
    \begin{equation} \label{equ:discrete sobolev inequaltiy}
        \|v_h\|_{L^{\infty}(\Omega)} \le C \|v_h\|_h. 
    \end{equation}
\end{theorem}
\begin{proof} First, notice that $H^2(\Omega) \hookrightarrow L^{\infty}(\Omega)$ and $E_hv_h \in H^2(\Omega)\cap H^1_0(\Omega)$. Then, by the triangle inequality and the inverse inequality \cite{brenner2008mathematical, ciarlet2002}, we can deduce that
    \begin{align*}
        \|v_h\|_{L^{\infty}(\Omega)} &\le \|E_hv_h\|_{L^{\infty}(\Omega)} + \|v_h - E_hv_h\|_{L^{\infty}(\Omega)} \\
        & \le C |E_hv_h|_{H^2(\Omega)} + C h^{-\frac{d}{2}} \|v_h - E_hv_h\|_{L^{2}(\Omega)} \\
        & \le C|v_h|_{H^2(\T_h)} + C |v_h - E_hv_h|_{H^2(\Omega)} + C h^{-\frac{d}{2}} \|v_h - E_hv_h\|_{L^{2}(\Omega)} \\
        & \le C \|v_h\|_h,
    \end{align*}
   where we have used \cite[Lemma 3]{neilan2019discrete} in the last inequality.
\end{proof}

\section{Linearization and finite element approximation}\label{sec:lin}
In this section, we consider the linearization of the model equation \eqref{equ: model equations} near $u^{\epsilon}$ and investigate its finite element approximation. It serves the building block to prove the existence, uniqueness, and error analysis of the numerical methods for the fourth order semi-linear PDE \eqref{equ: model equations}.

To linearize the model equation \eqref{equ: model equations}, we use the following identity from \cite{caffarelli1997properties}:
\begin{equation} \label{equ:identity}
	\det(D^2(u^{\eps} + tv)) = \det(D^2u^{\eps}) + t\,\mbox{tr}(\Phi^{\eps}D^2v) + \cdots+t^n \det(D^2v),
\end{equation}
where 
\begin{equation}
	\Phi^{\eps}:= \mbox{cof}(D^2u^{\eps}). 
\end{equation}
  By differentiating \eqref{equ:identity} at $t=0$, the linearized form of the operator $M^{\eps}(u^{\eps}):= \eps \Delta^2 u^{\eps} - \det(D^2u^{\eps})$ reads as
 \begin{equation}\label{equ:linearoperator}
	L_{u^{\eps}}(v):= \eps\Delta^2v - \Phi^{\eps}:D^2v = \eps\Delta^2v - \text{div}(\Phi^{\eps}\nabla v),
\end{equation}
where we have used the divergence-free property of the cofactor matrix $\Phi^{\eps}$  \cite[Lemma 7]{neilan2010nonconforming}. In the above equation, $A:B$ denotes the Frobenius inner product of two $d\times d$ matrices $A$ and $B$.

We focus on the approximation of  the following linearized equation 
\begin{subequations}\label{equ: Linearization of the model equations}
	\begin{align}
		L_{u^{\eps}}(v) = \varphi & \quad \text{in $\Omega$}, \\
		v = 0 & \quad \text{on $\partial\Omega$}, \\
		\Delta v = \psi & \quad \text{on $\partial\Omega$},
	\end{align}
\end{subequations}
where  $\varphi \in V^0_{*}$, $\psi \in H^{-\frac{1}{2}}(\partial\Omega)$, and $V^{0}_{*}$ is the dual space of $V^0$.

To define the finite element method for \eqref{equ: Linearization of the model equations}, we begin by converting \eqref{equ: Linearization of the model equations} into its weak form. Multiplying both sides of the equations by a test function $w \in V^0$ and applying the divergence theorem, we obtain the variational  formulation for \eqref{equ: Linearization of the model equations} as follows: Find $v \in V_0$ such that
\begin{equation}\label{equ: weak formulation of the linearization}
	a(v,w) = \langle\varphi, w \rangle + \eps(\psi, \nabla w\cdot{\bm n})_{L^2(\partial\Omega)} \quad \forall w \in V_0,
\end{equation}
where 
\begin{equation} \label{equ:bilinearform}
	a(v,w) = \eps(\Delta v, \Delta w)_{L^2(\Omega)} + (\Phi^{\eps} \nabla v, \nabla w)_{L^2(\Omega)}. 
\end{equation}

Under Assumption \ref{ass:convex assumption}, where $u^{\eps}$ is strictly convex, it follows that the matrix $\Phi^{\eps}$ is positive definite. Consequently, we can demonstrate that
\begin{equation*}
	a(v,v) \ge C \eps \|v\|_{H^2(\Omega)}^2. 
\end{equation*}
According to the Lax-Milgram lemma \cite{evans2010},  the variational problem \eqref{equ: weak formulation of the linearization} is well-posed.

 Define the discrete bilinear form $a_h(\cdot, \cdot)$ as
 \begin{equation}\label{equ:discrete bilinear form}
 	 \begin{aligned}
 	 	   a_h(v_h,w_h) & = \eps\left( \Delta v_h, \Delta w_h   \right)_{L^2(\mathcal{T}_h)} - \eps \sum_{ F \in F_h^i} \left( \la \Delta v_h \ra, \lj \nabla w_h \rj\right)_{L^2(F)} \\
 	 	 & \quad     
 	 	 - \eps \sum_{ F \in F_h^i} \left( \la \Delta w_h \ra, \lj \nabla v_h \rj\right)_{L^2(F)} 
       - \left( \Phi^{\eps} : D^2 v_h,  w_h   \right)_{L^2(\mathcal{T}_h)}. 
 	 \end{aligned}
 \end{equation}
In light of the discrete Miranda-Talenti estimate, we stabilize the discrete bilinear form by the jump of the flux.  We define the stablized bilinear form $A_h^{\sigma}(\cdot, \cdot)$ as 
    \begin{equation}\label{equ:stablized discrete bilinear form}
        A_h^{\sigma}(v_h,w_h) = a_h(v_h,w_h) + \sigma ( \eps + \eps^{-3})  \sum_{F \in \mathcal{F}_h^i} h_{F}^{-1} \left(  \lj \nabla v_h  \rj,  \lj \nabla w_h \rj   \right)_{L^2(\mathcal{F})},
    \end{equation}
where  $\sigma = \mathcal{O}(1)$ is a penalty parameter independent of $h$ and $\epsilon$.

The C0IP method of \eqref{equ: Linearization of the model equations} is to find $v_h\in V^0_h$ such that
\begin{equation}\label{equ: numerical scheme for linearization}
   A_h^{\sigma}(v_h,w_h) = \langle\varphi, w_h \rangle + \eps(\psi, \nabla w_h\cdot{\bm n})_{L^2(\partial\Omega)}, 
\end{equation}
for all $w_h \in V^0_h$.

For the C0IP method in \eqref{equ: numerical scheme for linearization}, we can establish the well-posedness. 
\begin{theorem}\label{thm:uniqueness of linearized model}
    There exists $\sigma_0 > 0$ such that for any $\sigma \geq \sigma_0$, there exists a unique $v_h \in V_h^0$ such that \eqref{equ: numerical scheme for linearization} holds.
\end{theorem}
\begin{proof} Boundedness of $A_h^{\sigma}$ can be easily proven using the trace inequality and the Cauchy-Schwartz inequality. We only need to establish the existence of $\sigma_0 > 0$ such that
    \begin{equation}\label{equ:coercivity}
     a_h(v_h,v_h) + \sigma ( \eps + \eps^{-3})  \sum_{F \in \mathcal{F}_h^i} h_{F}^{-1} \|  \lj \nabla v_h  \rj \|_{L^2(\mathcal{F})}^2 \ge C \eps \|v_h\|_h^2,
    \end{equation}
     for any $\sigma \geq \sigma_0$. 
    Using  integration by parts, we can deduce that 
    \begin{equation*}
        \begin{aligned}
            a_h(v_h,v_h) = &\eps\left( \Delta v_h, \Delta v_h   \right)_{L^2(\mathcal{T}_h)} - 2\eps \sum_{ F \in F_h^i} \left( \la \Delta v_h \ra, \lj \nabla v_h \rj\right)_{L^2(F)}  \\
            &\quad + (\Phi^{\eps} \nabla v_h, \nabla v_h   )_{L^2(\T_h)} - \sum_{F \in F_h^i} \left(  \lj \Phi^{\eps} \nabla v_h \rj, v_h  \right)_{L^2(F)}, 
        \end{aligned}
    \end{equation*}
Applying Holder's inequality and the trace inequality, we obtain
    \begin{align*}
                & \Big|  \sum_{F \in \mathcal{F}_h^i} \left( \la \Delta v_h \ra, \lj \nabla v_h  \rj\right)_{L^2(F)}   \Big| \\
                \le &\left( \sum_{F \in \mathcal{F}_h^i} h_{F} \| \la \Delta v_h \ra\|_{L^2(F)}^2  \right)^{1/2}\left( \sum_{F \in \mathcal{F}_h^i} h_{F}^{-1} \big\|   \lj \nabla v_h  \rj    \big\|_{L^2(F)}^2   \right)^{1/2} \\
                  \le& C  \left\|  \Delta v_h   \right\|_{L^2(\mathcal{T}_h)} \left( \sum_{F \in \mathcal{F}_h^i} h_{F}^{-1} \big\|   \lj \nabla v_h  \rj    \big\|_{L^2(F)}^2   \right)^{1/2} \\
                 \le &\frac{1}{8} \left\|  \Delta v_h   \right\|_{L^2(\mathcal{T}_h)}^2 + C \sum_{F \in \mathcal{F}_h^i} h_{F}^{-1} \big\|   \lj \nabla v_h \rj    \big\|_{L^2(F)}^2,
    \end{align*}
    where we have used the Young's inequality in the last inequality. 
   Similarly, we have 
    \begin{align*}
        &\Big|    \sum_{F \in F_h^i} \left(  \lj \Phi^{\eps} \nabla v_h \rj, v_h  \right)_{L^2(F)}    \Big| \\
        \le & \left( \sum_{F \in \mathcal{F}_h^i} h_{F} \|  v_h \|_{L^2(F)}^2  \right)^{1/2}\left( \sum_{F \in \mathcal{F}_h^i} h_{F}^{-1} \big\|   \lj \Phi^{\eps} \nabla v_h  \rj    \big\|_{L^2(F)}^2   \right)^{1/2} \\
         \le & C \eps^{-1} \|v_h\|_{L^2(\Omega)} \left( \sum_{F \in \mathcal{F}_h^i} h_{F}^{-1} \big\|   \lj \nabla v_h  \rj    \big\|_{L^2(F)}^2   \right)^{1/2} \\
         \le & \frac{\eps}{4} \|v_h\|_{L^2(\Omega)}^2 + C \eps^{-3}  \sum_{F \in \mathcal{F}_h^i} h_{F}^{-1} \big\|   \lj \nabla v_h  \rj    \big\|_{L^2(F)}^2.
    \end{align*}
Then, \eqref{equ:coercivity} follows by combining the above two estimates with the discrete Miranda-Talenti estimate \eqref{equ:discrete MT}. 
\end{proof}

\begin{remark}
    If we have a uniform lower bound on $\Phi^{\epsilon}$ such that 
    \begin{equation*}
        {\bm x}^{T}\Phi^{\eps} {\bm x}  \geq C \|{\bm x}\|^2  \quad \forall \eps >0, 
    \end{equation*}
    then $\eps^{-3}$ in \eqref{equ: numerical scheme for linearization} can be replaced by $\eps^{-2}$.
\end{remark}

Furthermore, we can get the following error estimate in discrete $H^2$ norm.

\begin{theorem}\label{thm: H^2 error estimate of linearization}
Let $v$ be the solution of \eqref{equ: weak formulation of the linearization} and $v_I$ be its interpolation in $V_h$. Suppose $v \in V^0\cap H^{s}(\Omega)$, then there holds
    \begin{equation} \label{equ: closeness error}
        \|v_h - v_I\|_h    \le C \left(1 + \eps^{-2} +  \sigma(1 + \eps^{-4})\right) h^l |v|_{H^s(\Omega)},
    \end{equation}
    and  
    \begin{equation} \label{equ: error of linearization problem}
        \|v - v_h\|_h    \le C \left(1 + \eps^{-2} +  \sigma(1 + \eps^{-4})\right) h^l |v|_{H^s(\Omega)}, 
    \end{equation}
    where $l = \min(k-1, s-2)$. 
\end{theorem}
\begin{proof}
   Let $w_h = v_h - v_I$. Then,  we have 
       \begin{equation} \label{equ:main term}
    	    \begin{aligned}
        A_h^{\sigma}(w_h,w_h) & = A_h^{\sigma}(v_h,w_h) - A_h^{\sigma}(v_I,w_h) \\
        & = \langle\varphi, w_h \rangle + \eps(\psi, \nabla w_h\cdot{\bm n})_{L^2(\partial\Omega)} - A_h^{\sigma}(v_I,w_h). 
    \end{aligned}
    \end{equation}
    Using the equation  \eqref{equ: Linearization of the model equations}, we have 
    \begin{equation} \label{equ:intermeidate result}
    	    \begin{aligned}
        \langle\varphi, w_h \rangle & = ( \eps \Delta^2 v -  \Phi^{\eps} : D^2v   , w_h)_{L^2(\Omega)} \\
        = &-\eps (\nabla \Delta v, \nabla w_h)_{L^2(\Omega)} -  (\Phi^{\eps}: D^2 v,  w_h)_{L^2(\Omega)} \\
         = &\quad \eps (\Delta v, \Delta w_h)_{L^2(\T_h)} - \eps \sum_{F \in \mathcal{F}_h^i} \left(  \Delta v, \lj \nabla w_h \rj\right)_{L^2(F)} \\
         & - \eps (\psi, \nabla w_h \cdot {\bm n})_{L^2(\partial\Omega)}  -  (\Phi^{\eps}: D^2 v,  w_h)_{L^2(\Omega)}.
    \end{aligned}
    \end{equation}
Substituting \eqref{equ:intermeidate result} into \eqref{equ:main term}, we can deduce that    
    \begin{align*}
        &A_h^{\sigma}(w_h,w_h)\\
         = & \eps (\Delta v, \Delta w_h)_{L^2(\T_h)} - \eps \sum_{F \in \mathcal{F}_h^i} \left(  \Delta v, \lj \nabla w_h \rj\right)_{L^2(F)}  -  (\Phi^{\eps}: D^2 v, w_h)_{L^2(\Omega)} - A_h^{\sigma}(v_I,w_h) \\
        = &\quad \eps (\Delta v - \Delta v_I, \Delta w_h)_{L^2(\T_h)} \\
        &- \eps \sum_{F \in \mathcal{F}_h^i} \left(  \Delta v - \la \Delta v_I \ra, \lj \nabla w_h  \rj\right)_{L^2(F)} - \eps \sum_{F \in \mathcal{F}_h^i} \left(  \la \Delta w_h \ra, \lj \nabla v - \nabla v_I  \rj\right)_{L^2(F)} \\
        & -  (\Phi^{\eps}: D^2( v - v_I),  w_h)_{L^2(\T_h)} -  \sigma ( \eps + \eps^{-3} ) \sum_{F \in \mathcal{F}_h^i} h_{F}^{-1} \left(  \lj \nabla v_I -\nabla v \rj,  \lj \nabla w_h  \rj   \right)_{L^2(F)}.
    \end{align*}
Using the trace inequality and \eqref{equ: approximaiton of interpolation}, we obtain:
    \begin{equation*}
    	    \begin{aligned}
        &|A_h^{\sigma}(w_h,w_h)| \\
        \le  & C \eps h^l|v|_{H^s(\Omega)} \|w_h\|_h + C \eps^{-1} h^l |v|_{H^s(\Omega)} \|w_h\|_h + C \sigma ( \eps + \eps^{-3} ) h^l |v|_{H^s(\Omega)} \|w_h\|_h \\
         \le  & C \eps \left(1 + \eps^{-2} +  \sigma(1 + \eps^{-4})\right)h^l |v|_{H^s(\Omega)} \|w_h\|_h. 
    \end{aligned}
    \end{equation*}
Combining the previous estimate with the coercivity \eqref{equ:coercivity}, we have established \eqref{equ: closeness error}. Moreover, \eqref{equ: error of linearization problem} directly follows from \eqref{equ: closeness error} and \eqref{equ: approximaiton of interpolation}.
\end{proof}

\section{$C^0$IP methods for the model equation} \label{sec:c0ip}
In this section, we formulate the C0IP equation for the nonlinear model equation \eqref{equ: model equations} and perform the error analysis.

\subsection{Numerical scheme}\label{ssec:ns}

The weak form of the model equation \eqref{equ: model equations} reads as: Seeking $u^{\eps} \in V^g$ such that 
\begin{equation}
	-\eps(\Delta u^{\eps}, \Delta v)_{L^2(\Omega)} + \Big( \det(D^2u^{\eps}), v \Big)_{L^2(\Omega)} = (f,v)_{L^2(\Omega)} - \eps \langle\eps, \nabla v \cdot {\bm n}\rangle_{L^2(\partial\Omega)}, \quad \forall v \in V^0.
\end{equation}

To facilicate the definition of numerical scheme, we introduce the new bilinear form $b_h^{\sigma}(\cdot, \cdot)$ as 
\begin{equation}
\begin{aligned}
	 b_h^{\sigma}(v_h,w_h) = &  \sigma (\eps + \eps^{-3})  \sum_{F \in \mathcal{F}_h^i} h_{F}^{-1} \left(  \lj \nabla v_h  \rj,  \lj \nabla w_h  \rj   \right)_{L^2(F)}  \\
    &- \eps \sum_{F \in \mathcal{F}_h^i} \left( \la \Delta v_h \ra, \lj \nabla w_h \rj\right)_{L^2(F)} - \eps \sum_{F \in \mathcal{F}_h^i} \left( \la \Delta w_h \ra, \lj \nabla v_h \rj\right)_{L^2(F)} .
\end{aligned}
\end{equation}
The new  C0IP method for the model problem is defined to  finding  
$u_h^{\eps} \in V_h^g$ such that 
\begin{equation}\label{equ: numerical scheme for model problem}
    \begin{aligned}
       & -\eps\Big(\Delta u_h^{\eps}, \Delta v_h \Big)_{L^2(\T_h)}  
         +  \big( \det(D^2u_h^{\eps}), v_h \big)_{L^2(\T_h)} -b_h^{\sigma}(u^{\epsilon}_h, v_h) \\
          = &  (f,v_h)_{L^2(\Omega)} - \eps \langle\eps, \nabla v_h\cdot {\bm n}\rangle_{L^2(\partial\Omega)}, 
    \end{aligned}
\end{equation}
for all $v_h \in V_h^0$. 

It's worth noting that the term $\big( \det(D^2u_h^{\eps}), v_h \big)_{L^2(\T_h)}$ is nonlinear. In the next subsection, we will establish the existence and uniqueness of the solution for \eqref{equ: numerical scheme for model problem}.

\subsection{Well-posedness and $H^2$ error estimate for the numerical scheme}
In this section, we establish the well-posedness of \eqref{equ: numerical scheme for model problem} and simultaneously prove the convergence rate using the combined fixed point and linearization techniques as in \cite{feng2009mixed, neilan2010nonconforming}.

We first define a linear operator $T: V_h^{g} \to V_h^{g}$. For any given $v_h \in V_h^{g}$, let $T(v_h) \in V_h^{g}$ denote the solution of the following problem:
\begin{equation}\label{equ:fixed poiont map}
	\begin{aligned}
	&A_h^{\sigma}( v_h - T(v_h), w_h ) =  \eps\Big(\Delta v_h, \Delta w_h \Big)_{L^2(\T_h)} -  \big( \det(D^2v_h), w_h \big)_{L^2(\T_h)} \\
	& + b_h^{\sigma} (v_h , w_h) + (f,w_h)_{L^2(\Omega)} - (\eps^2, \nabla w_h\cdot {\bm n})_{L^2(\partial\Omega)}, \quad \forall w_h \in V_h^0.
\end{aligned}
\end{equation}

Theorem \ref{thm:uniqueness of linearized model} tells that the map $T$ is well-defined. It is not hard to see that the fixed point of $T$ is the solution of \eqref{equ: numerical scheme for model problem} and vice versa. We now proceed to prove the existence of such a fixed point in the vicinity of the interpolation of $u^{\epsilon}$. For this purpose, we define a neighborhood near $u^{\epsilon}_I$ as
\begin{equation}\label{equ:neighborhood}
	\mathbb{B}_h(\rho) := \{ v_h \in V_h^g: \ \|v_h - u_I^{\eps}\|_{h} \le \rho \}.
\end{equation}

We commence our proof by establishing the following lemma
\begin{lemma}\label{lem:sit on the ball}
Suppose $u^{\epsilon} \in H^s(\Omega)$. Then, the following estimate holds
	\begin{equation}
        \begin{aligned}
		\|u_I^{\eps} - T(u_I^{\eps})\|_{h} 
         \le C_1 \eps^{-1} \left( 1  + \eps^{-d} + \sigma(1 + \eps^{-4}) \right)h^{\min(k-1, s-2)}\|u^{\epsilon}\|_{H^s(\Omega)}
        \end{aligned}
	\end{equation}
	for some constant $C_1 > 0$.
\end{lemma}
\begin{proof}
Let $w_h^{\eps} = u_I^{\eps} - T(u_I^{\eps})$. Then, we have
\begin{equation}
	\begin{aligned}
		A_h^{\sigma}( w_h^{\eps}, w_h^{\eps} ) = & \eps\Big( \Delta u_I^{\eps}, \Delta w_h^{\eps}\Big)_{L^2(\Omega)} - \big( \det(D^2u_I^{\eps}), w_h^{\eps} \big)_{L^2(\Omega)} \\
	& + b_h^{\sigma} (u_I^{\eps} , w_h^{\eps}) + (f,w_h^{\eps})_{L^2(\Omega)} - (\eps^2, \nabla w_h^{\eps}\cdot {\bm n})_{L^2(\partial\Omega)}.
	\end{aligned}
\end{equation}
Using the definition of the model equation \eqref{equ: model equations}, we can derive that:
\begin{equation}
	\begin{aligned}
		(f,w_h^{\eps})_{L^2(\Omega)} & = -\eps(\Delta^2 u^{\eps}, w_h^{\eps})_{L^2(\Omega)} + (\det(D^2u^{\eps}), w_h^{\eps})_{L^2(\Omega)} \\
		& = \eps (\nabla \Delta u^{\eps}, \nabla w_h^{\eps})_{L^2(\Omega)} 
		  + (\det(D^2u^{\eps}), w_h^{\eps})_{L^2(\Omega)} \\
		& = -\eps (\Delta u^{\eps}, \Delta w_h^{\eps})_{L^2(\T_h)} + \eps\sum_{F \in \mathcal{F}_h^{i}}(\Delta u^{\eps}, \lj \nabla w_h^{\eps}\rj )_{L^2(F)} \\
		 & + (\eps^2, \nabla w_h^{\eps}\cdot {\bm n})_{L^2(\partial\Omega)} 
		  + (\det(D^2u^{\eps}), w_h^{\eps})_{L^2(\Omega)}.
	\end{aligned}
\end{equation}
Combining the above two equalities, utilizing the trace inequality, and using the same estimation techniques in Theorem \ref{thm: H^2 error estimate of linearization}, along with the interpolation result \eqref{equ: approximaiton of interpolation}, we obtain
	\begin{equation}\label{equ: bound of a for fixed point}
	\begin{aligned}
		A_h^{\sigma}( w_h^{\eps}, w_h^{\eps} )  = & \eps\Big(\Delta u_I^{\eps} - \Delta u^{\eps}, \Delta w_h^{\eps}\Big)_{L^2(\T_h)} + \eps\sum_{F \in \mathcal{F}_h^{i}}(\Delta u^{\eps}, \lj \nabla w_h^{\eps} \rj )_{L^2(F)} \\
		& \quad + \big( \det(D^2u^{\eps}) - \det(D^2u_I^{\eps}), w_h^{\eps} \big)_{L^2(\T_h)}  +  b_h^{\sigma} (u_I^{\eps} , w_h^{\eps})\\
		\le & C \eps \left(1 + \sigma(1 + \eps^{-4}) \right) h^{\min(k-1, s-2)} |u^{\eps}|_{H^s(\Omega)}\|w_h^{\eps}\|_{h}  \\
		&+\big( \det(D^2u^{\eps}) 
		 - \det(D^2u_I^{\eps}), w_h^{\eps} \big)_{L^2(\T_h)}.
		\end{aligned}
\end{equation}

To bound the last term in \eqref{equ: bound of a for fixed point}, we apply the Mean Value Theorem and deduce that 
	\begin{equation*}
		\big( \det(D^2u^{\eps}) - \det(D^2u_I^{\eps}), w_h^{\eps} \big)_{L^2(\T_h)} = \Big(\Psi^{\eps}:(D^2u^{\eps} - D^2u_I^{\eps}), w_h^{\eps}\Big)_{L^2(\T_h)}, 
	\end{equation*}
	where $\Psi^{\eps} = \text{cof}(D^2u^{\eps} - \theta(D^2u^{\eps} - D^2u_I^{\eps} ))$ for some $\theta \in [0,1]$. 
Using \cite[Lemma 4.1]{feng2011conforming} and the priori estimate \eqref{equ: priori bounds for u^eps}, we have 
	\begin{equation*}
		\|\Psi^{\eps}\|_{L^{\infty}(\Omega)} \le C\|D^2u^{\eps}\|_{L^{\infty}(\Omega)}^{d-1} \le C \eps^{1-d}.
	\end{equation*}
Combining the above two estimates with \eqref{equ: bound of a for fixed point}, we obtain:
	\begin{align*}
		A_h^{\sigma}( w_h^{\eps}, w_h^{\eps} ) \le  &C \eps \left(1 + \sigma(1 + \eps^{-4}) \right) h^{\min(k-1, s-2)} |u^{\eps}|_{H^s(\Omega)}\|w_h^{\eps}\|_{h}  \\
		& + C \eps^{1-d} h^{\min(k-1, s-2)} |u^{\eps}|_{H^s(\Omega)}\|w_h^{\eps}\|_{h}.
	\end{align*}
	The desired result is derived from the coercivity \eqref{equ:coercivity}. 

\end{proof}

Next, we introduce a lemma to demonstrate the contraction property of the operator $T$
\begin{lemma}\label{lem:contraction property}
	For any $v_h, w_h \in \mathbb{B}_h(\rho)$ there holds
	\begin{equation}
		\|T(v_h) - T(w_h)\|_{h} \le C(h,\rho,\eps) \|v_h - w_h\|_{h},
	\end{equation}
	where $C(h,\rho,\eps) = C_2 \eps^{1 -d}\Big( \eps^{-1} h^{\min(k-1, s-2)}  + \rho \Big)$, for some constant $C_2 > 0$. 
\end{lemma}

\begin{proof}
	Let $z_h = v_h - w_h$. Using the definition of $T$ and \eqref{equ:stablized discrete bilinear form}, we have, for any $\eta_h \in V_h^0$:
	\begin{align*}
		&A_h^{\sigma}( T(v_h) - T(w_h), \eta_h )\\  = & \Big( \det(D^2 v_h) - \det(D^2 w_h), \eta_h   \Big)_{L^2(\T_h)}  - \Big( \Phi^{\eps}: D^2 z_h,  \eta_h \Big)_{L^2(\T_h)} \\
		 = & (\Psi_h:D^2 z_h, \eta_h)_{L^2(\T_h)}  - \Big( \Phi^{\eps}: D^2 z_h,  \eta_h \Big)_{L^2(\T_h)},
	\end{align*}
	where we have utilized the Mean Value Theorem in the last equality, and $\Psi_h = \text{cof}(D^2w_h + \theta(D^2v_h - D^2w_h) )$ for some $\theta \in [0,1]$.
	Next, we rewrite the above expression as
	\begin{align*}
		A_h^{\sigma}( T(v_h) - T(w_h), \eta_h )= \Big( (\Psi_h - \Phi^{\eps}):D^2z_h, \eta_h \Big)_{L^2(\T_h)} .
	\end{align*}
	To estimate $A_h^{\sigma}( T(v_h) - T(w_h), \eta_h )$, we first quote the following estimate from \cite[Lemma 4.3]{feng2011conforming}
		\begin{align*}
		\|\Psi_h - \Phi^{\eps}\|_{L^2(\T_h)}  \le C \eps^{2 -d} (h^{\min(k-1, s-2)}  \|u^{\eps}\|_{H^2(\Omega)} + \rho ).
	\end{align*}
	Using the above estimate, the Cauchy-Schwartz inequality, and the discrete Sobolev inequality \eqref{equ:discrete sobolev inequaltiy}, we can deduce that
	\begin{align*}
		A_h^{\sigma}( T(v_h) - T(w_h), \eta_h ) & \le C \eps^{2 -d} (h^{\min(k-1, s-2)}  \|u^{\eps}\|_{H^3(\Omega)} + \rho ) \|z_h\|_{h} \|\eta_h \|_{L^{\infty}(\Omega)} \\
		&\le C \eps^{2 -d} (h^{\min(k-1, s-2)}  \|u^{\eps}\|_{H^3(\Omega)} + \rho ) \|z_h\|_{h} \|\eta_h \|_{h}.
	\end{align*}
The desired result follows by setting $\eta_h = T(v_h)-T(w_h)$ and applying the 
coercivity \eqref{equ:coercivity}.
\end{proof}

Now, we are well-prepared to establish the well-posedness of our numerical scheme and derive the $H^2$ error estimate.
\begin{theorem}
Suppose $u^{\epsilon}\in H^s(\Omega)$. Then, there exists a unique $u_h^{\eps} \in V_h^g$ satisfying \eqref{equ: numerical scheme for model problem}, with the error estimates given by
    \begin{equation}
        \|u^{\eps} - u_h^{\eps}\|_h \le C_4 \eps^{-1} \left( 1 + \eps^{-d} + \sigma(1 + \eps^{-4}) \right)h^{\min(k-1, s-2)}\|u\|_{H^s(\Omega)}.
    \end{equation}
\end{theorem}
\begin{proof}
    Let $\rho_0 = 2C_1 \eps^{-1} \left( 1 + \eps^{-d} + \sigma(1 + \eps^{-4}) \right)h^{\min(k-1, s-2)}\|u\|_{H^s(\Omega)}$. Then, the constant $C(h, \rho_0, \epsilon)$ becomes
    \begin{equation}
    \begin{aligned}    	
    	C(h, \rho_0, \epsilon) =  C_2 \eps^{ -d}\Big( 1   +  2C_1  \left( 1  + \eps^{-d} + \sigma(1 + \eps^{-4}) \right)\|u\|_{H^2(\Omega)} \Big)h^{\min(k-1, s-2)}.
    \end{aligned}
    \end{equation}
We choose $h_0$ such that for $h<h_0$, we have
    \begin{equation}
    	C(h, \rho_0, \epsilon) < \frac{1}{2}
    \end{equation}
According to Lemma \ref{lem:contraction property}, for $v_h,w_h \in \mathbb{B}_h(\rho_0)$ and $h<h_0$, we then have
 	\begin{align} \label{equ:conntraction inequality}
		\|T(v_h) - T(w_h)\|_{h} \le \frac{1}{2} \|v_h - w_h\|_{h}. 
	\end{align} 
	Furthermore, Lemma \ref{lem:sit on the ball} and \eqref{equ:conntraction inequality} imply that for any $z_h\in \mathbb{B}_h(\rho_0)$
	\begin{align*}
		&\|T(z_h) - u_I^{\eps}\|_{h} \\ 
		 \le & \|T(z_h) - T(u_I^{\eps})\|_{h} + \|T(u_I^{\eps}) - u_I^{\eps}\|_{h} \\
		 \le &\frac{1}{2}  \|z_h - u_I^{\eps}\|_{h} +   C_1 \eps^{-1} \left( 1 + \eps^{-d} + \sigma(1 + \eps^{-4}) \right)h^{\min(k-1, s-2)}\|u\|_{H^s(\Omega)}  \\
		 \le & \frac{1}{2} \rho_0 +  \frac{1}{2}\rho_0  = \rho_0.
	\end{align*}
	It means that $T(z_h)\in \mathbb{B}_h(\rho_0)$.
By the Brouwer's Fixed Point Theorem \cite{evans2010}, there exists a unique $u_h^{\eps} \in \mathbb{B}_h(\rho_0)$ such that $T(u_h^{\eps}) = u_h^{\eps}$, which means $u_h^{\eps}$ is a unique solution of the numerical scheme. To obtain the $H^2$ error estimate, we apply the triangle inequality and deduce that
	\begin{equation}
		\begin{aligned}
			\|u^{\epsilon} - u^{\epsilon}_h\|_h \le  & \|u^{\epsilon} - u^{\epsilon}_I\|_h  + \|u^{\epsilon}_I - u^{\epsilon}_h\|_h \\
		 \le 	&  Ch^{\min(k-1, s-2)}\|u\|_{H^s(\Omega)} + \rho_0 \\
		 \le & C_4 \eps^{-1} \left( 1 + \eps^{-d} + \sigma(1 + \eps^{-4}) \right)h^{\min(k-1, s-2)}\|u\|_{H^s(\Omega)}.
		\end{aligned}
	\end{equation}
This concludes the proof.
	\end{proof}

\section{Numerical experiments}\label{sec:num}
In this section, we provide a series of numerical examples to demonstrate the performance of the proposed finite element methods and validate the theoretical results.  We use Newton's method as the nonlinear solver. All of the tests given below are computed on the domain $\Omega = (0,1)^d$.

\subsection{Two dimensional numerical experiments} In this subsection, we consider the two dimensional numerical experiments. 

\begin{table}[htb!]
\centering
\caption{Numerical errors of numerical test I with fixed $h = 0.01$.  }\label{tab:varepsilon}
\resizebox{0.99\textwidth}{!}{
\begin{tabular}{|c|c|c|c|c|c|c|c|c|c|}
\hline 
Degree &$\epsilon$ & $\|u-u^h\|_{L^2(\Omega)} $ & Order & $\|u-u^h\|_{H^1(\Omega)}$&Order  & $\|u-u^h\|_{H^2(\mathcal{T}_h)}$ & Order \\ \hline
\multirow{5}{*}{k=2}     & 5.00e-01 &1.39e-01&--&6.48e-01&--&3.53e+00&--\\ \cline{2-8}
 & 2.50e-01 &1.08e-01&0.37&5.12e-01&0.34&3.17e+00&0.15\\ \cline{2-8}
 & 1.25e-01 &8.85e-02&0.28&4.29e-01&0.26&2.94e+00&0.11\\ \cline{2-8}
 & 5.00e-02 &5.39e-02&0.54&2.78e-01&0.47&2.50e+00&0.18\\ \cline{2-8}
 & 2.50e-02 &3.21e-02&0.75&1.82e-01&0.61&2.16e+00&0.21\\ \cline{2-8}
 & 1.25e-02 &1.79e-02&0.85&1.16e-01&0.66&1.85e+00&0.22\\ \cline{2-8}
 & 5.00e-03 &7.76e-03&0.91&6.16e-02&0.69&1.48e+00&0.24\\ \cline{2-8}
 & 2.50e-03 &3.98e-03&0.96&3.74e-02&0.72&1.24e+00&0.26\\ \hline
\multirow{5}{*}{k=3}     & 5.00e-01 &1.39e-01&--&6.49e-01&--&3.55e+00&--\\ \cline{2-8}
 & 2.50e-01 &1.08e-01&0.37&5.12e-01&0.34&3.20e+00&0.15\\ \cline{2-8}
 & 1.25e-01 &8.86e-02&0.28&4.29e-01&0.26&2.97e+00&0.11\\ \cline{2-8}
 & 5.00e-02 &5.39e-02&0.54&2.79e-01&0.47&2.52e+00&0.18\\ \cline{2-8}
 & 2.50e-02 &3.22e-02&0.74&1.83e-01&0.61&2.19e+00&0.21\\ \cline{2-8}
 & 1.25e-02 &1.80e-02&0.84&1.16e-01&0.65&1.88e+00&0.22\\ \cline{2-8}
 & 5.00e-03 &7.88e-03&0.90&6.23e-02&0.68&1.52e+00&0.23\\ \cline{2-8}
 & 2.50e-03 &4.13e-03&0.93&3.84e-02&0.70&1.28e+00&0.24\\ \hline
\end{tabular}}
\end{table}

\subsubsection{Numerical Example I}
In this example, we investigate the approximation of the perturbed equation \eqref{equ: model equations} to the fully nonlinear  Monge-Amp\`{e}re equation \eqref{equ: MA equation}. We choose $f(x,y) = (1+x^2+y^2)e^{(x^2+y^2)/2}$ and $g(x,y) = e^{(x^2+y^2)/2}$ so that  $u(x,y) =  e^{(x^2+y^2)/2}$ is the unique classical solution of \eqref{equ: MA equation}.

In the first test, we fix the mesh size $h = 0.01$ and run the tests for varying $\epsilon$. Given the small $h$,  $\|u- u^{\epsilon}_h\|$ be considered an accurate estimate of $\|u- u^{\epsilon}\|$. Table \ref{tab:varepsilon} provides the results obtained from the simulation using quadratic and cubic elements. It can be seen from the data in Table \ref{tab:varepsilon} that the $\|u-u^h\|_{L^2(\Omega)} \approx \mathcal{O}(\epsilon)$, $\|u-u^h\|_{H^1(\Omega)}\approx\mathcal{O}(\epsilon^{0.75})$, and $\|u-u^h\|_{H^2(\mathcal{T}_h)}\approx\mathcal{O}(\epsilon^{0.25})$ for both cases. It suggests that $u^{\epsilon}$ converges to $u$ in $H^2$ norm at rate of $\mathcal{O}(\epsilon^{0.25})$.

Next, we investigate the relationship between  $\epsilon$  and  $h$  to determine the "optimal" choice of $h$ that he global error $\|u- u^{\epsilon}_h\|$ is the same order  as that of $\|u- u^{\epsilon}\|$. We fit the constant in $y = \beta \epsilon^{\alpha}$ using the data, where $\alpha = 0.25$ for the discrete $H^2$ norm, $\alpha = 0.75$ for the $H^1$ norm, and $\alpha = 1$ for the $L^2$ norm.  The numerical results for quadratic and cubic elements are presented in Figure \ref{fig:errrel} with $h = \sqrt{\epsilon}$.  From the graph, we can see that the fitted curves match the data. This implies that $h = \sqrt{\epsilon}$ is the best choice of $h$ in terms of $\epsilon$.

\begin{figure}[!h]
\captionsetup[subfigure]{labelformat=empty}
   \centering
   \subcaptionbox{$L^2$ error \label{fig:l2err_p2}}
  {\includegraphics[width=0.49\textwidth]{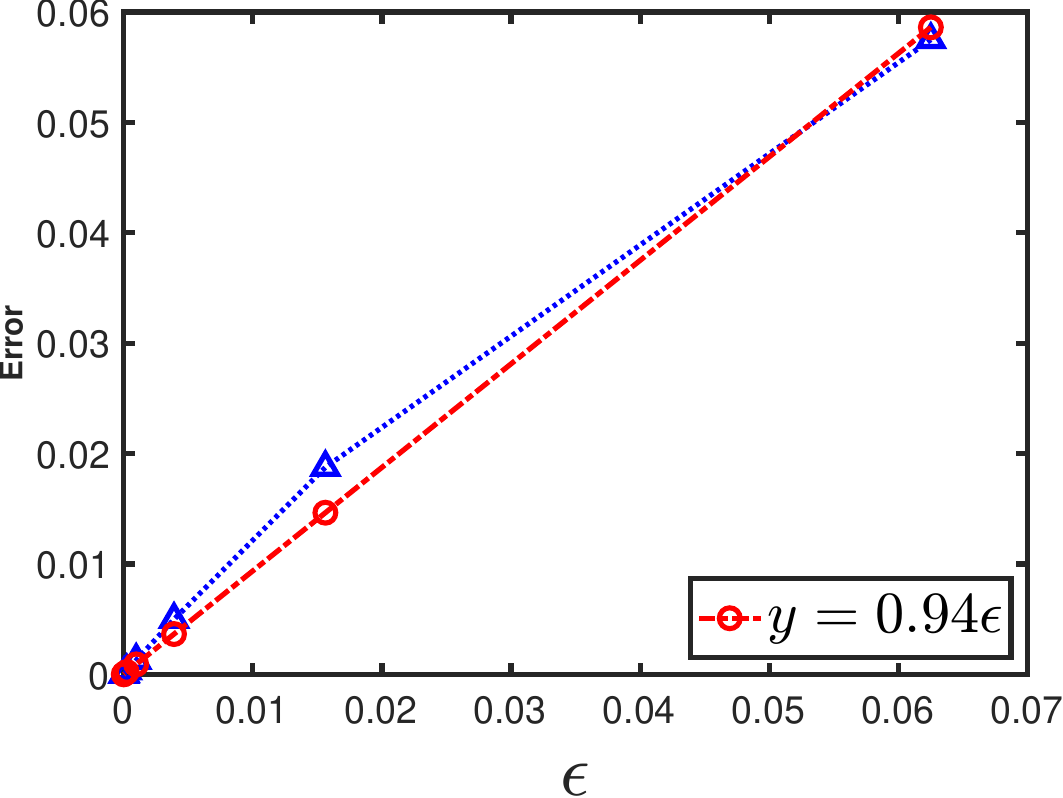}}
  \subcaptionbox{ $L^2$ error\label{fig:l2err_p3}}
   {\includegraphics[width=0.49\textwidth]{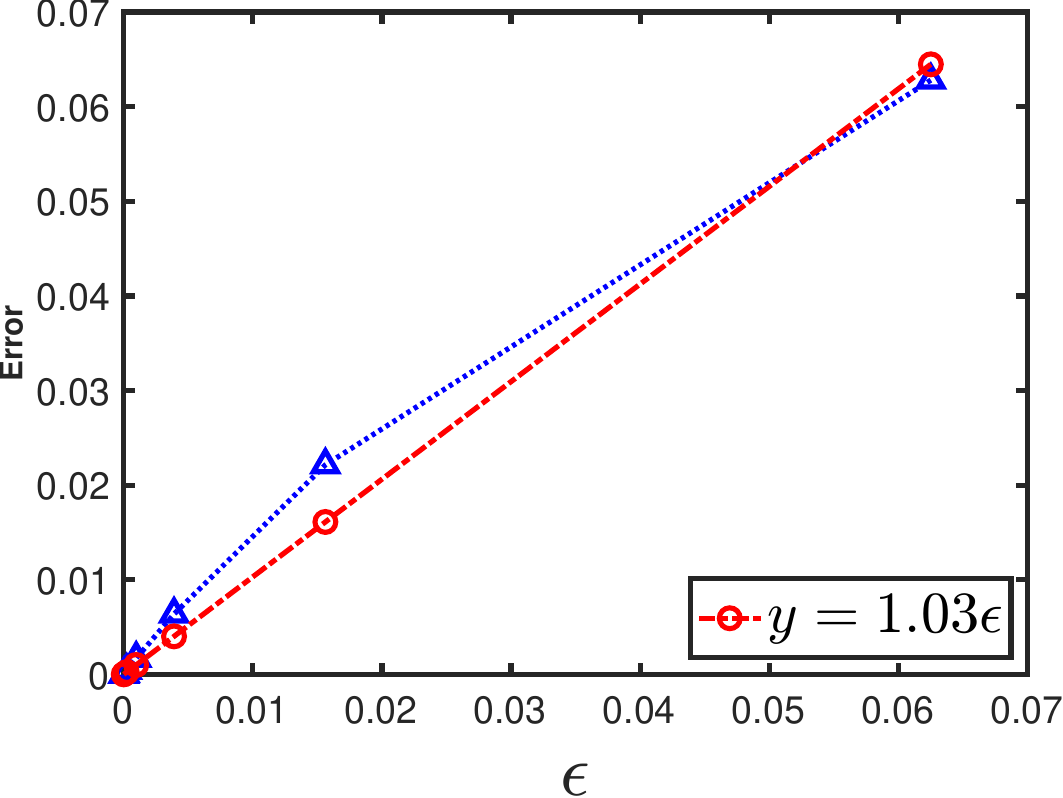}}
    \subcaptionbox{$H^1$ error \label{fig:h1err_p2}}
  {\includegraphics[width=0.49\textwidth]{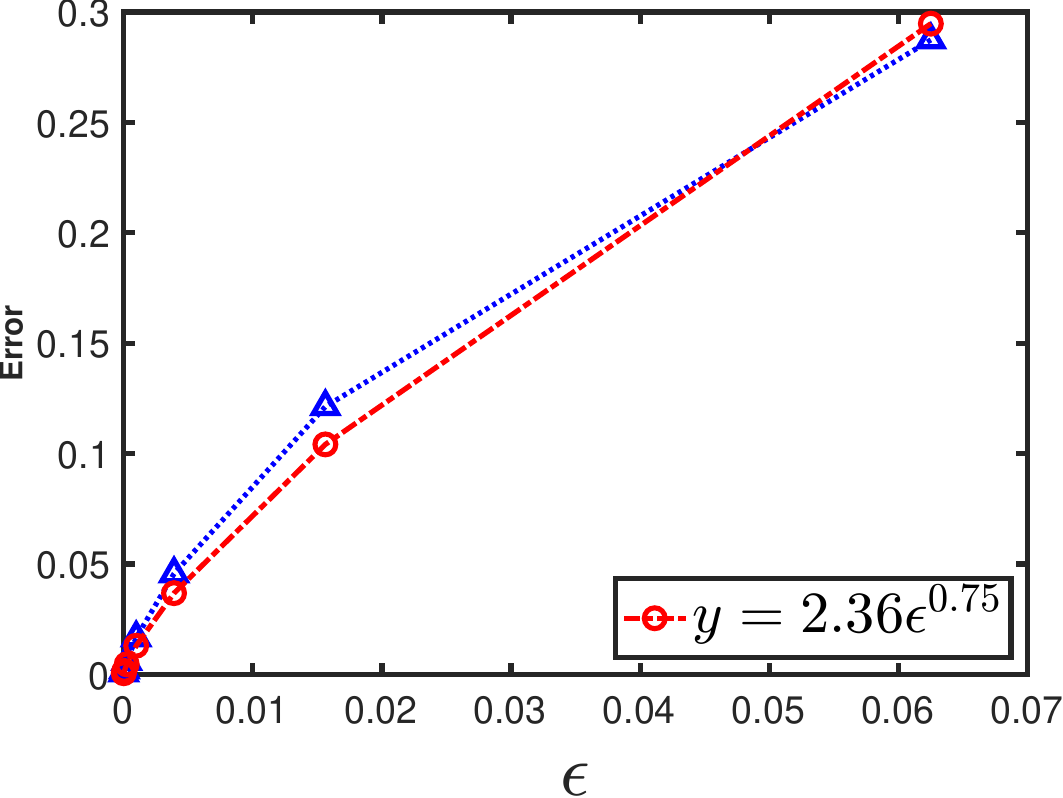}}
  \subcaptionbox{ $H^1$ error\label{fig:h1err_p3}}
   {\includegraphics[width=0.49\textwidth]{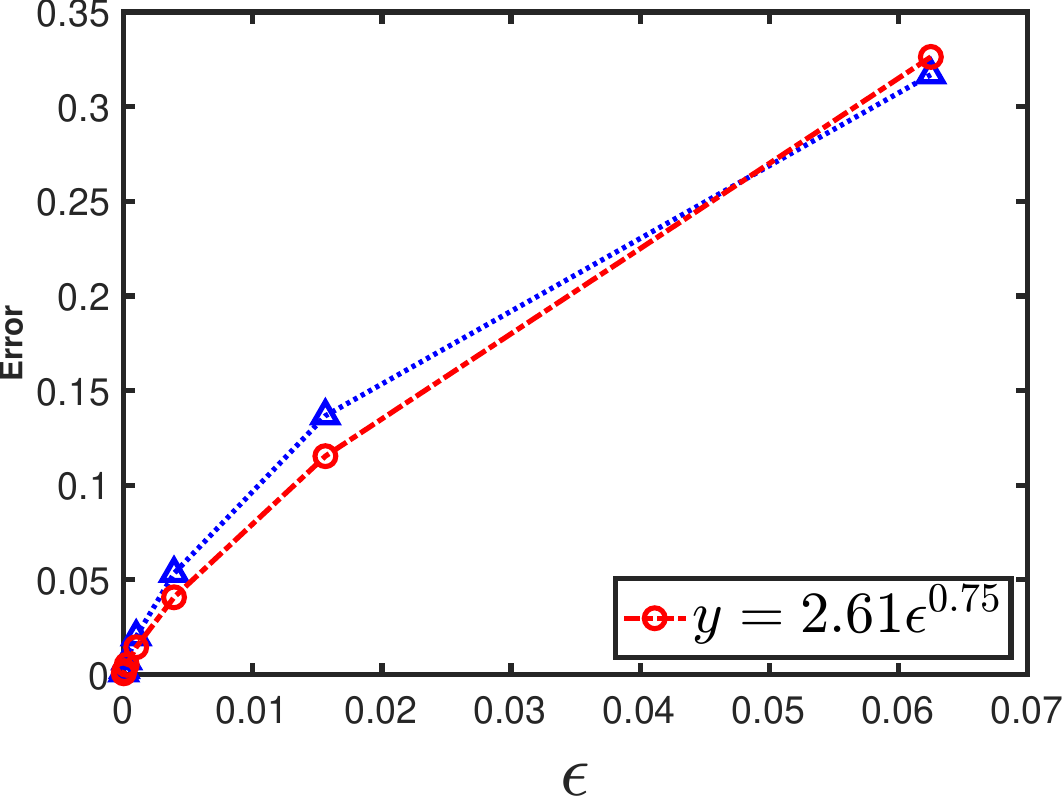}}
      \subcaptionbox{Discrete $H^2$ error \label{fig:h2err_p2}}
  {\includegraphics[width=0.49\textwidth]{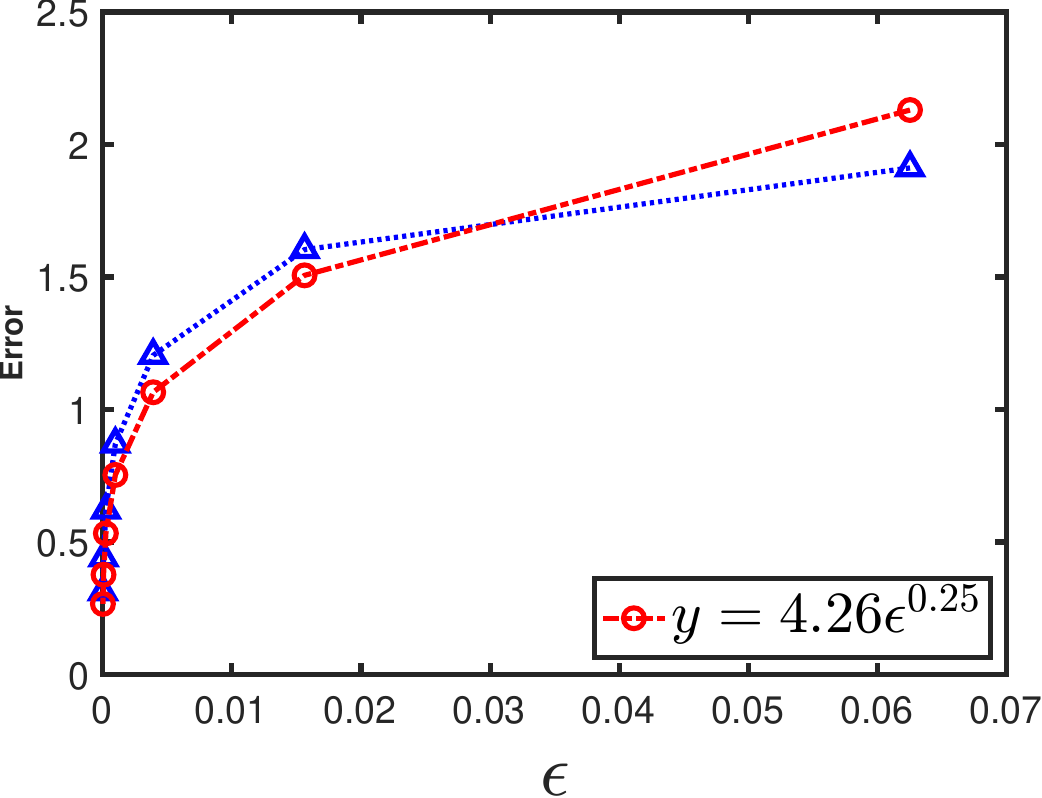}}
  \subcaptionbox{ Discrete $H^2$ error\label{fig:h2err_p3}}
   {\includegraphics[width=0.49\textwidth]{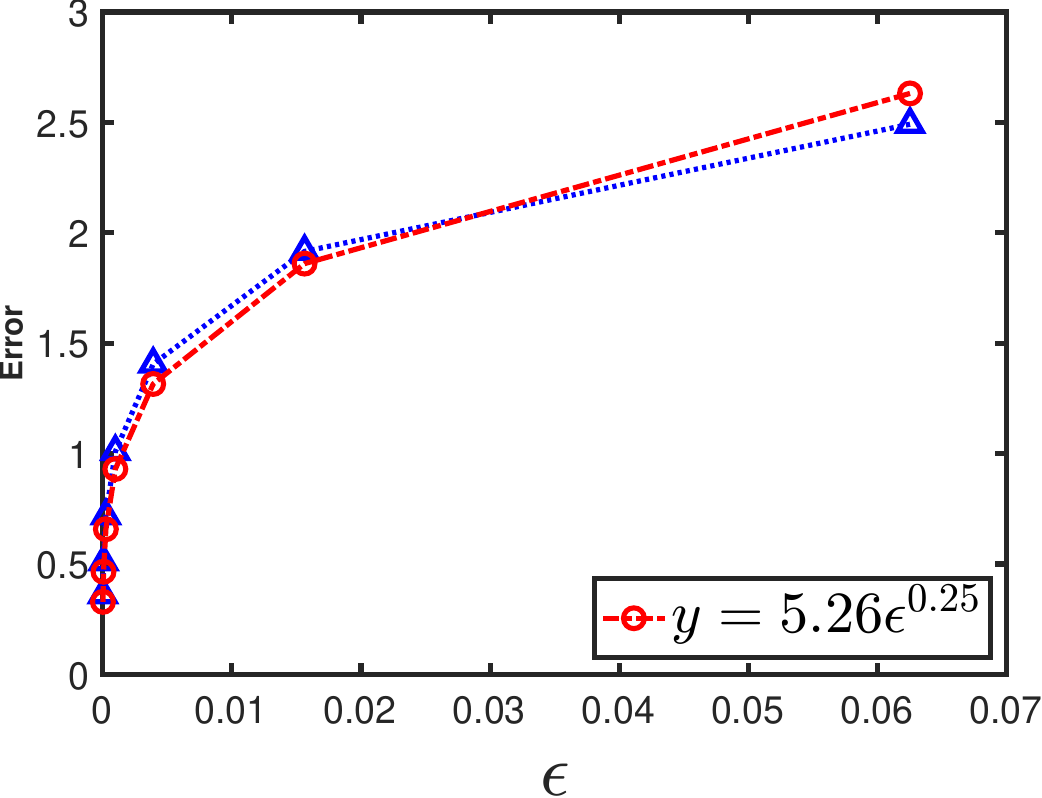}}
   \caption{Plot of error with respect to $\epsilon$.  The first column  is the numerical result for quadratic element  and the second column  is the numerical result of cubic element.}\label{fig:errrel}
\end{figure}

\subsubsection{Numerical example II} In this example, we analyze the rate of convergence for $\| u^{\epsilon} - u^{\epsilon}_h\|$ for fixed $\epsilon = 0.01$, where $u^{\epsilon}$ is the solution of \eqref{equ: model equations}.  We solve the equation \ref{equ: model equations} with the boundary condition $\Delta u^{\epsilon} = \psi^{\epsilon} $ instead of $\Delta u^{\epsilon} = \epsilon $. We choose $f^{\epsilon} = x^2y^2-4\epsilon$, $g^{\epsilon} = \frac{1}{12}(x^4+y^4)$, and $\psi^{\epsilon} = x^2 + y^2$ to fit the exact solution $u^{\epsilon} = \frac{1}{2}(x^4+y^4)$.

\begin{table}[htb!]
\centering
\caption{Numerical errors of numerical test II with fixed $\epsilon  = 0.01$.  }\label{tab:varh}
\resizebox{0.99\textwidth}{!}{
\begin{tabular}{|c|c|c|c|c|c|c|c|c|c|}
\hline 
Degree &$h$ & $\|u-u^h\|_{L^2(\Omega)} $ & Order & $\|u-u^h\|_{H^1(\Omega)}$&Order  & $\|u-u^h\|_{H^2(\mathcal{T}_h)}$ & Order \\ \hline
\multirow{5}{*}{k=2}  
  & 1/8 &3.98e-04&--&2.05e-03&--&5.92e-02&--\\ \cline{2-8}
 & 1/16 &8.04e-05&2.31&4.36e-04&2.23&2.95e-02&1.00\\ \cline{2-8}
 & 1/32 &1.88e-05&2.09&1.04e-04&2.07&1.47e-02&1.00\\ \cline{2-8}
 & 1/64 &4.63e-06&2.03&2.57e-05&2.02&7.37e-03&1.00\\ \cline{2-8}
 & 1/128 &1.15e-06&2.01&6.40e-06&2.01&3.68e-03&1.00\\   \hline 
\multirow{5}{*}{k=3}  
   & 1/8 &2.52e-05&--&1.16e-04& -- &1.64e-03& -- \\ \cline{2-8}
 & 1/16 &5.94e-06&2.09&2.71e-05&2.10&4.08e-04&2.01\\ \cline{2-8}
 & 1/32 &1.46e-06&2.02&6.64e-06&2.03&1.02e-04&2.00\\ \cline{2-8}
 & 1/64 &3.64e-07&2.01&1.65e-06&2.01&2.55e-05&2.00\\ \cline{2-8}
 & 1/128 &9.16e-08&1.99&4.16e-07&1.99&6.37e-06&2.00\\  \hline
\end{tabular}}
\end{table}

The numerical results  for quadratic and cubic elements are summarized in Table \ref{tab:varh}. 
Looking at the table, it is apparent that the optimal convergence order of $\mathcal{O}(h^{k-1})$ can be observed for both quadratic and cubic elements. This finding is consistent with the theoretical results presented in our theorem. 
What's interesting about the data in this table is that we only observe a suboptimal order of $\mathcal{O}(h^{k-1})$ for the $L^2$ error and $H^1$ error for cubic elements, contrasting with the optimal convergence rate of $\mathcal{O}(h^k)$ for quadratic elements. This is despite the error for cubic elements being several digits in magnitude less than the corresponding error for quadratic elements.

\subsubsection{Numerical Example III} In the numerical example, we test the capability of the proposed methods for computing viscosity solution. 
 Similar to \cite{neilan2010nonconforming}, we choose $f=1$ and $g=0$. 
As demonstrated in \cite{guti2016ma}, the  Monge-Amp\`{e}re equation \eqref{equ: MA equation} admits a unique viscosity solution but does not have a classical solution.  We simulated the numerical solution using quadratic element with $\epsilon = 0.005$ and $h = 1/32$.  The simulation result is visulized in Figure \ref{fig:vissol}.  As can be seen from the figure,  the proposed numerical methods can track the viscosity convex solution.

\begin{figure}[!h]
   \centering
  \includegraphics[width=0.65\textwidth]{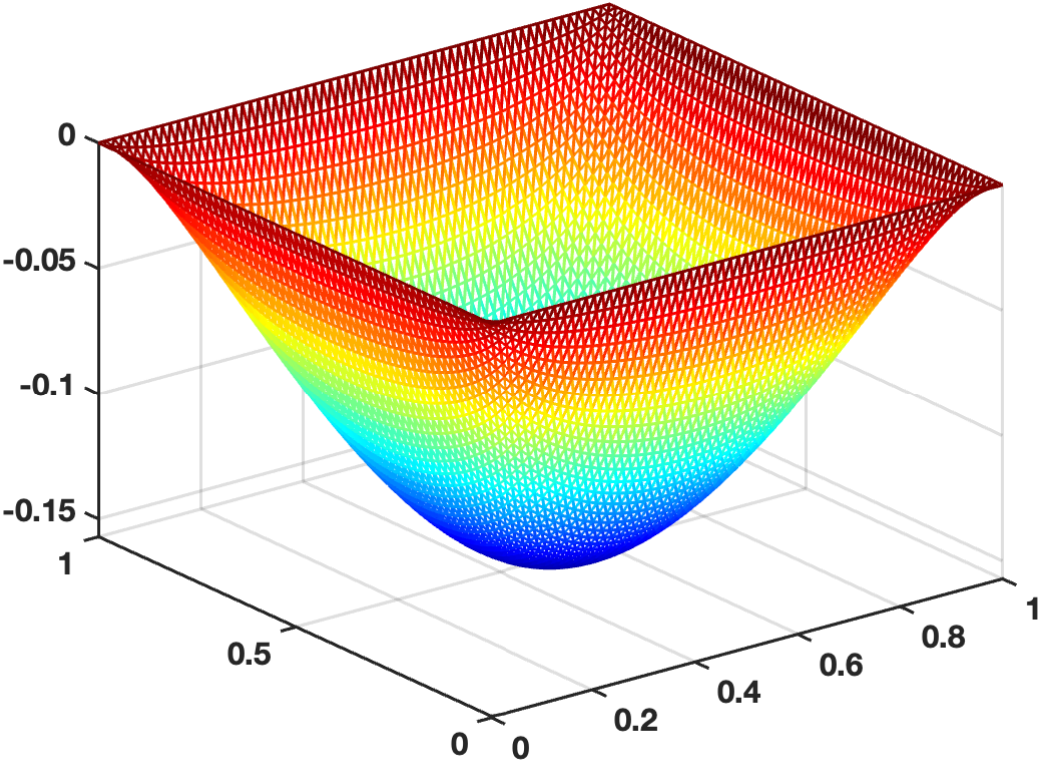}
   \caption{Plot of the computed viscosity solution using quadratic element with $\epsilon=0.005$ and $h = 1/64$ in 2D.  }
   \label{fig:vissol}
\end{figure}

\subsection{Three dimensional numerical experiments}
In this subsection, we present the numerical result in three dimensional case. 

\subsubsection{Numerical example IV}

In this example, we investigate the approximation of the perturbed equation \eqref{equ: model equations} to the fully nonlinear  Monge-Amp\`{e}re equation \eqref{equ: MA equation}. We select $f(x,y,z) = (1+x^2+y^2+z^2)e^{(x^2+y^2+z^2)/2}$ and $g(x,y,z) = e^{(x^2+y^2+z^2)/2}$ to fit the unique classical solution  $u(x,y,z) =  e^{(x^2+y^2+z^2)/2}$.

In the numerical experiment,  we test the convergence rate of $u_h^{\epsilon}$ to $u$ for fixed $h$ and variable $\epsilon$.  Table \eqref{tab:varepsilon3d} displays the numerical results for quadratic and cubic element on uniform tetrahedral triangulation on the unit cube. Similar to two dimensional case, we can observe that $\|u_h^{\epsilon} - u\|_{L^2(\Omega)} \approx \mathcal{O}(\epsilon)$, $\|u_h^{\epsilon} - u\|_{L^2(\Omega)} \approx \mathcal{O}(\epsilon^{0.75})$, $\|u_h^{\epsilon} - u\|_{L^2(\Omega)} \approx \mathcal{O}(\epsilon^{0.25})$ for both quadratic element and cubic element.

\begin{table}[htb!]
\centering
\caption{Numerical errors of numerical test IV with fixed $h = 0.05$.  }\label{tab:varepsilon3d}
\resizebox{0.99\textwidth}{!}{
\begin{tabular}{|c|c|c|c|c|c|c|c|c|c|}
\hline 
Degree &$\epsilon$ & $\|u-u^h\|_{L^2(\Omega)} $ & Order & $\|u-u^h\|_{H^1(\Omega)}$&Order  & $\|u-u^h\|_{H^2(\mathcal{T}_h)}$ & Order \\ \hline
\multirow{8}{*}{k=2}   
 & 5.00e-01 &1.25e-01& -- &7.23e-01& -- &4.67e+00& -- \\ \cline{2-8}
 & 2.50e-01 &9.17e-02&0.45&5.43e-01&0.41&4.08e+00&0.20\\ \cline{2-8}
 & 1.25e-01 &6.94e-02&0.40&4.24e-01&0.36&3.68e+00&0.15\\ \cline{2-8}
 & 5.00e-02 &3.86e-02&0.64&2.60e-01&0.53&3.07e+00&0.20\\ \cline{2-8}
 & 2.50e-02 &2.24e-02&0.79&1.69e-01&0.62&2.62e+00&0.23\\ \cline{2-8}
 & 1.25e-02 &1.24e-02&0.85&1.08e-01&0.65&2.19e+00&0.26\\ \cline{2-8}
 & 5.00e-03 &5.44e-03&0.90&5.82e-02&0.67&1.63e+00&0.32\\ \cline{2-8}
 & 2.50e-03 &2.84e-03&0.94&3.58e-02&0.70&1.23e+00&0.40\\ \hline
 \multirow{8}{*}{k=3}   
 & 5.00e-01 &1.49e-01& -- &8.74e-01& -- &6.22e+00& -- \\ \cline{2-8}
 & 2.50e-01 &1.07e-01&0.48&6.46e-01&0.44&5.42e+00&0.20\\ \cline{2-8}
 & 1.25e-01 &7.77e-02&0.45&4.90e-01&0.40&4.83e+00&0.16\\ \cline{2-8}
 & 5.00e-02 &4.16e-02&0.68&2.93e-01&0.56&3.99e+00&0.21\\ \cline{2-8}
 & 2.50e-02 &2.38e-02&0.80&1.89e-01&0.63&3.40e+00&0.23\\ \cline{2-8}
 & 1.25e-02 &1.31e-02&0.86&1.20e-01&0.66&2.86e+00&0.25\\ \cline{2-8}
 & 5.00e-03 &5.78e-03&0.90&6.43e-02&0.68&2.23e+00&0.27\\ \cline{2-8}
 & 2.50e-03 &3.05e-03&0.92&3.98e-02&0.69&1.81e+00&0.30\\ \hline
\end{tabular}}
\end{table}

\subsection{Numerical example V}

In the numerical example, we test the the rate of convergence of $u-u_h$  for fixed $\epsilon$ in this dimensional case. For this purpose, we replace the boundary condition  $\Delta u^{\epsilon} = \epsilon $ of  \eqref{equ: model equations} by $\Delta u^{\epsilon} = \psi^{\epsilon} $. We choose $f^{\epsilon} = 36x^2z^2-24\epsilon$, $g^{\epsilon} = \frac{1}{2}(x^4+y^2+z^4)$, and $\psi^{\epsilon} = 1+6x^2 + 6z^2$. It is easy to verify that  the exact solution is $u^{\epsilon} = \frac{1}{2}(x^4+z^2+y^4)$.

In this numerical test, we select $\epsilon = 0.005$ and the numerical results are reported in Table \ref{tab:varh3d}.  In term of the error of $u^{\epsilon} - u^{\epsilon}_h$ in $H^2$ norm, we can observe otpimal order of $\mathcal{O}(h^k)$ for both quadratic and cubic elements. Similarly to two dimensional case, we can only observe $\mathcal{O}(h^2)$ convergence for the $L^2$ and $H^1$ error of $u^{\epsilon} - u^{\epsilon}_h$   for cubic element. What stands out in the table is that we can observe $\mathcal{O}(h^3)$ order convergence for the $L^2$ error of  $u^{\epsilon} - u^{\epsilon}_h$ for the quadratic element.

\begin{table}[htb!]
\centering
\caption{Numerical errors of numerical test V with fixed $\epsilon  = 0.01$.  }\label{tab:varh3d}
\resizebox{0.99\textwidth}{!}{
\begin{tabular}{|c|c|c|c|c|c|c|c|c|c|}
\hline 
Degree &$h$ & $\|u-u^h\|_{L^2(\Omega)} $ & Order & $\|u-u^h\|_{H^1(\Omega)}$&Order  & $\|u-u^h\|_{H^2(\mathcal{T}_h)}$ & Order \\ \hline
\multirow{4}{*}{k=2}  
 & 1/3 &2.07e-03&--&4.10e-02&--&9.38e-01&--\\ \cline{2-8}
 & 1/6 &2.61e-04&2.98&1.02e-02&2.00&4.71e-01&0.99\\ \cline{2-8}
 & 1/12 &3.45e-05&2.92&2.55e-03&2.01&2.36e-01&1.00\\ \cline{2-8}
 & 1/24 &5.37e-06&2.68&6.36e-04&2.00&1.18e-01&1.00\\ \hline
\multirow{4}{*}{k=3}  
 & 1/3 &1.40e-04&--&2.40e-03&--&7.17e-02&--\\ \cline{2-8}
 & 1/6 &1.43e-05&3.30&3.61e-04&2.73&1.75e-02&2.04\\ \cline{2-8}
 & 1/12 &2.55e-06&2.49&5.04e-05&2.84&4.35e-03&2.01\\ \cline{2-8}
 & 1/24 &6.12e-07&2.06&7.16e-06&2.82&1.09e-03&2.00\\ \hline
\end{tabular}}
\end{table}

\subsubsection{Numerical example VI}
 In the numerical example, we test the capability of the proposed methods for computing viscosity solution in three dimensional case.  Similar to \cite{feng2009vanishing},  we choose $f=1$ and $g=0$. 
As demonstrated in \cite{guti2016ma}, the  Monge-Amp\`{e}re equation \eqref{equ: MA equation} admits a unique viscosity solution but does not have a classical solution.  We simulated the numerical solution using quadratic element with $\epsilon = 0.005$ and $h = 1/32$.  In Figure \ref{fig:slice}, we plot the $x$-slices (left graph) and $y$-slices (right graph) of the computed solution. Again,  the proposed numerical methods can track the viscosity convex solution.


\begin{figure}[!h]
\captionsetup[subfigure]{}
   \centering
   \subcaptionbox{ \label{fig:xslice}}
  {\includegraphics[width=0.49\textwidth]{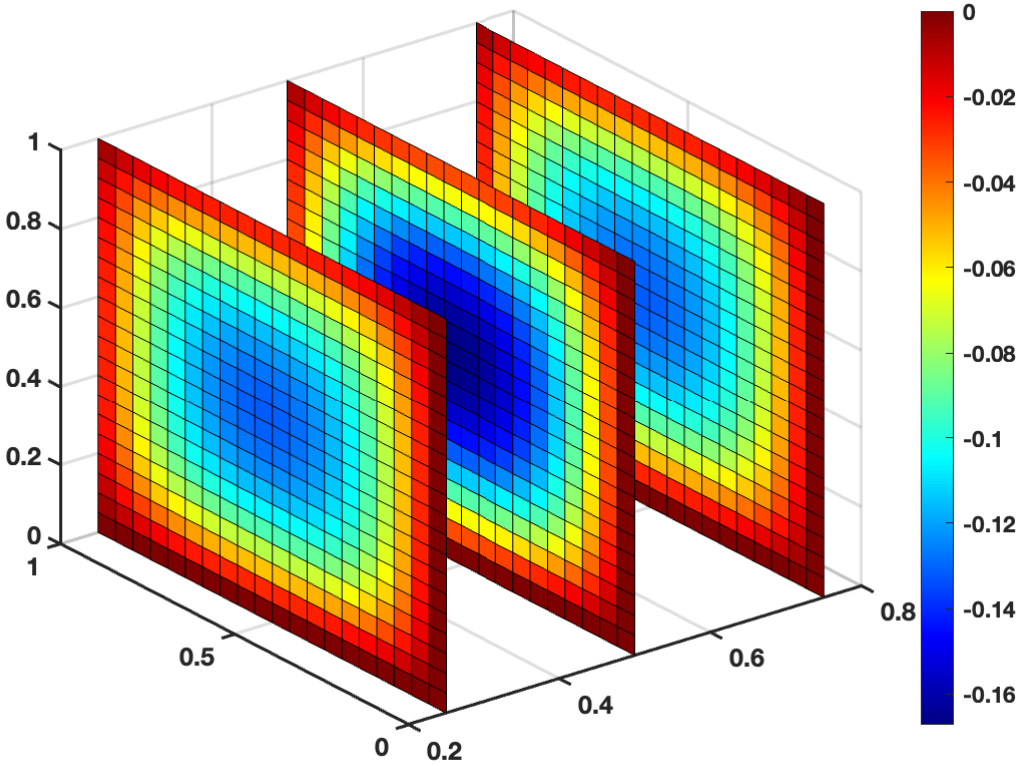}}
  \subcaptionbox{\label{fig:yslice}}
   {\includegraphics[width=0.49\textwidth]{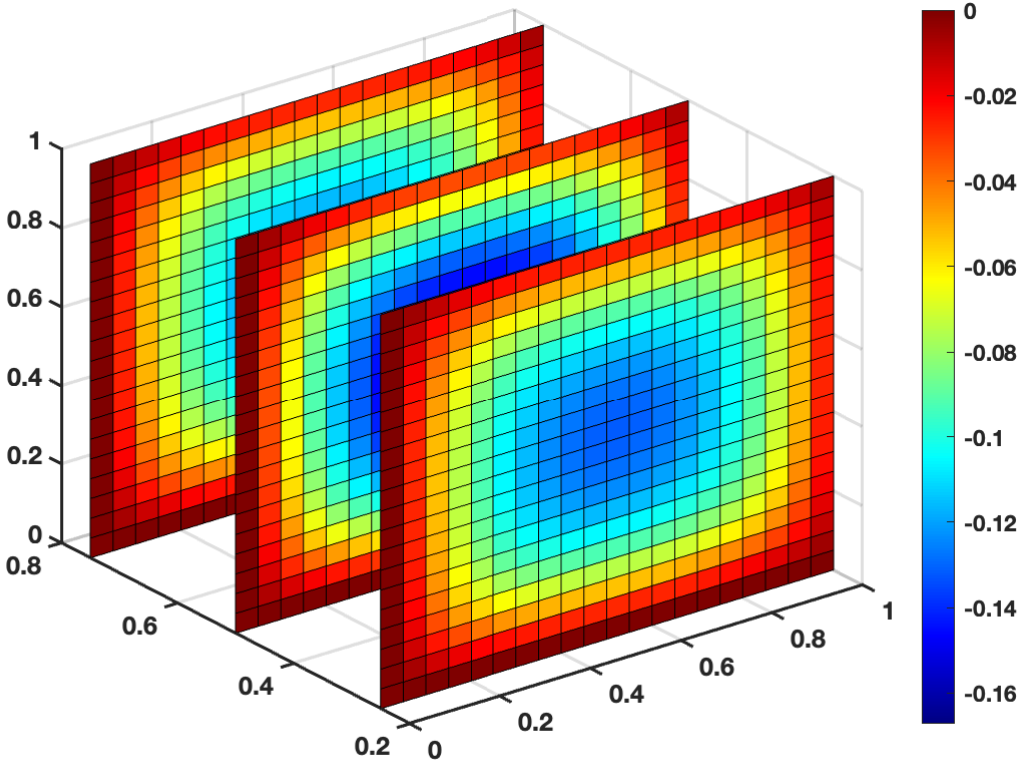}}
   \caption{Plot of the computed viscosity solution using quadratic element with $\epsilon=0.005$ and $h = 1/32$ in 3D: (a) $x$-slices at $x =0.25, 0.5, 0.75$; (b) $y$-slices at $y =0.25, 0.5, 0.75$.}\label{fig:slice}
\end{figure}

\section{Conclusion} \label{sec:con}
In this paper, we propose and analyze a new C0IP method for computing the viscosity solution of the fully nonlinear Monge-Amp\`{e}re equations. The key idea is to apply the discrete Miranda-Talenti estimate. We prove optimal error estimates in $H^2$-norm. A series of benchmark examples are provided to demonstrate the theoretical results. 

\section*{Acknowledgments}
 This work was supported in part by the Andrew Sisson Fund, Dyason Fellowship, the Faculty Science Researcher Development Grant of the University of Melbourne, and the NSFC grant 12131005.

\bibliographystyle{siamplain}
\bibliography{references}

\end{document}